\documentclass[a4paper,intlimits]{article}
\usepackage{latexsym,amsfonts,amssymb,amsmath,amsthm}
\date{}
\newtheorem{proposition}{Proposition}[section]
\newtheorem{theorem}[proposition]{Theorem}
\newtheorem{lemma}[proposition]{Lemma}
\newtheorem*{LA}{Lemma A}
\numberwithin{equation}{section}

\DeclareMathSymbol{\shortrightarrow}{\mathrel}{AMSa}{"4B}

\DeclareMathOperator*{\interior}{int}

\DeclareMathOperator*{\esssup}{ess\,sup}
\newcommand{\pw}{a.e.\ }
\newcommand{\boldv}{\ensuremath{\mathbf{v}}}
\newcommand{\CalA}{\ensuremath{\mathcal{A}}}
\newcommand{\CalB}{\ensuremath{\mathcal{B}}}
\newcommand{\CalK}{\ensuremath{\mathcal{K}}}
\newcommand{\CalL}{\ensuremath{\mathcal{L}}}
\newcommand{\CalN}{\ensuremath{\mathcal{N}}}
\newcommand{\CalP}{\ensuremath{\mathcal{P}}}
\newcommand{\gothP}{\ensuremath{\mathfrak{P}}}
\newcommand{\Gb}{\ensuremath{\mathcal{G}_{b}}}
\newcommand{\barGb}{\ensuremath{\overline{\mathcal{G}}_{b}}}
\newcommand{\Id}{\ensuremath{\mathrm{Id}}}
\newcommand{\JM}{Mierczy\'nski}
\newcommand{\Polacik}{Pol\'a\v{c}ik}
\newcommand{\Terescak}{Tere\v{s}\v{c}\'ak}

\newcommand{\JDDE}{J. Dynam. Differential Equations}
\newcommand{\JDE}{J. Differential Equations}
\newcommand{\JMAA}{J. Math. Anal. Appl.}
\newcommand{\e}{\ensuremath{\mathbf{e}}}
\newcommand{\eg}{e.g.\ }
\newcommand{\resp}{resp.\ }
\newcommand{\reals}{\ensuremath{\mathbb{R}}}
\newcommand{\integers}{\ensuremath{\mathbb{Z}}}
\newcommand{\naturals}{\ensuremath{\mathbb{N}}}
\newcommand{\dD}{\ensuremath{\partial_{\mathcal{D}}\Omega}}
\newcommand{\dR}{\ensuremath{\partial_{\mathcal{R}}\Omega}}
\newcommand{\clomega}{\ensuremath{\overline{\Omega}}}
\newcommand{\domega}{\ensuremath{\partial\Omega}}
\newcommand{\Ce}{\ensuremath{C(\clomega)}}
\newcommand{\Cee}{\ensuremath{C_{\e}(\clomega)}}
\newcommand{\Ceone}{\ensuremath{C^{1}(\clomega)}}
\newcommand{\Cetwo}{\ensuremath{C^{2}(\clomega)}}
\newcommand{\Cezero}{\ensuremath{C_{0}(\clomega)}}
\newcommand{\Cezeroone}{\ensuremath{C^{1}_{0}(\clomega)}}
\newcommand{\bt}{\ensuremath{b\cdot t}}
\newcommand{\Nbs}{\ensuremath{{\CalN}(b\cdot s)}}
\newcommand{\Bb}{\ensuremath{\mathbb{B}}}
\newcommand{\Ltwo}{\ensuremath{L^{2}(\Omega)}}
\newcommand{\Rn}{\ensuremath{{\reals}^{n}}}
\newcommand{\Sn}{\ensuremath{{\mathbb{S}}^{n-1}}}
\newcommand{\0}{\ensuremath{\{0\}}}
\newcommand{\abs}[1]{\ensuremath{\lvert#1\vert}}
\newcommand{\norm}[1]{\ensuremath{\lVert#1\rVert}}
\newcommand{\Cenorm}[1]{\ensuremath{\lVert#1\rVert_{\Ce}}}
\newcommand{\Ceonenorm}[1]{\ensuremath{\lVert#1\rVert_{\Ceone}}}
\newcommand{\Ceenorm}[1]{\ensuremath{\lVert#1\rVert_{\e}}}
\newcommand{\Ceestnorm}[1]{\ensuremath{\lVert#1\rVert_{\Cee^{*}}}}
\renewcommand{\ge}{\ensuremath{\geqslant}}
\renewcommand{\le}{\ensuremath{\leqslant}}

\newcommand{\leone}{\ensuremath{\mathrel{\le_{1}}}}
\newcommand{\lessone}{\ensuremath{\mathrel{<_{1}}}}
\newcommand{\llone}{\ensuremath{\mathrel{\ll_{1}}}}

\newcommand{\lee}{\ensuremath{\mathrel{\le_{\e}}}}
\newcommand{\Le}{\ensuremath{\mathrel{<_{\e}}}}
\newcommand{\LL}{\ensuremath{\mathrel{\ll_{\e}}}}
\newcommand{\gee}{\ensuremath{\mathrel{\ge_{\e}}}}
\newcommand{\Ge}{\ensuremath{\mathrel{>_{\e}}}}
\newcommand{\GG}{\ensuremath{\mathrel{\gg_{\e}}}}
\newcommand{\CalW}{\ensuremath{\mathcal{W}}}
\newcommand{\CalS}{\ensuremath{\mathcal{S}}}
\newcommand{\CalT}{\ensuremath{\mathcal{T}}}

\newcommand{\incmap}{\ensuremath{\mathrel{{\subset}%
\hspace{-0.4em}%
\raisebox{-0.67ex}{\ensuremath{\shortrightarrow}}}}}
\begin{document}
\title{Globally positive solutions of linear parabolic partial differential equations of second order with Dirichlet boundary
conditions\footnotemark{*}
\footnotetext[0]{
\copyright\ 1998. This manuscript version is made available under the CC-BY-NC-ND 4.0 license http://creativecommons.org/licenses/by-nc-nd/4.0/}
\footnotetext[0]{Published in {\em Journal of Mathematical Analysis and Applications} \textbf{226}(2) (1998), pp. 326--347.}
\footnotetext[0]{https://doi.org/10.1006/jmaa.1998.6065}
\footnotetext[1]{Research supported
by KBN grant 2 P03A 076 08.}
}
\author{Janusz Mierczy\'nski\\
\small Institute of Mathematics\\
\small Technical University of Wroc{\l}aw\\
\small Wybrze\.ze Wyspia\'nskiego 27\\
\small PL-50-370 Wroc{\l}aw\\
\small Poland\\
\small mierczyn@banach.im.pwr.wroc.pl}
\maketitle
The purpose of this paper is to study globally positive
solutions of a linear nonautonomous parabolic partial
differential equation (PDE) of second order
\begin{equation*}
u_{t}=\sum_{i,j=1}^{n}a_{ij}(x)
\frac{\partial^{2}u}{\partial x_{i}\partial x_{j}}
+\sum_{i=1}^{n}a_{i}(x)\frac{\partial u}{\partial x_{i}}
+a_{0}(t,x)u,
\quad t\in\reals, x\in\Omega,
\tag{E}
\end{equation*}
where $\Omega\subset\Rn$ is a bounded domain, complemented
with the homogeneous Dirichlet boundary conditions
\begin{equation*}
\tag{BC}
u(t,x)=0,\quad t\in\reals, x\in\domega.
\end{equation*}
The main result (Theorem~\ref{main-theorem}) states that the
set of solutions to (E)+(BC) that are defined and of
constant sign for all $t\in\reals$ and $x\in\Omega$ is a
one-dimensional vector space.

S.-N. Chow, K. Lu and J. Mallet-Paret in their 1995
paper~\cite{Ch-L-MP2} were the first to address the issue
for $n=1$ (see also their earlier paper~\cite{Ch-L-MP1}).
In~fact, they proved much more:  There is an invariant
decomposition of the vector space of all global (in time)
solutions into the direct sum of countably many
one-dimensional subspaces indexed by positive integers
$l$.  They obtained a characterization of the $l$-th
subspace as the set of all global solutions for which the
(Matano) lap number is constantly $l-1$.

For a general space dimension $n$, the present author
in~\cite{JM2} has considered the problem of characterizing
globally positive solutions of a linear second order
parabolic PDE under Robin (regular oblique) boundary
conditions.  The reasoning from that paper, however, fails
to carry over to general boundary conditions.

This paper is organized as follows.  The starting point is
to show (in Section~\ref{Basic}) that the solution operator
for (E)+(BC) and all limits of its time translates possess
a sufficiently regular kernel (Green's function).  Our
principal results are included in Section~\ref{Positive}.
Here, by applying ideas from the theory of linear
skew-product semidynamical systems on Banach bundles, we
show that globally positive solutions must be contained in
some invariant subbundle of dimension one.  In the Appendix
generalizations to other boundary conditions are given.

\section{Basic properties of solutions}
\label{Basic}
We write $\Rn=\{(x_{1},\dots,x_{n}):x_{i}\in\reals\}$ for
the $n$\nobreakdash-\hspace{0pt}dimensional real vector
space with the Euclidean norm $\norm{\cdot}$, and $\Sn$ for
the unit sphere in $\Rn$.

For a function $f:S\to Y$, where $S\subset\Rn$ and $Y$ is a
Banach space, we use the notation
$D_{i}:=\partial/{\partial}x_{i}$.  If $f$ is defined on
$S\subset\reals\times\Rn=\{(t,x_{1},\dots,x_{n})\}$, we
write $D_{t}:=\partial/{\partial}t$.  As~usual, $D_{ij}$
means $D_{i}D_{j}$.  The derivatives may be considered in
the classical sense, or in the distribution sense,
depending on the context.  For $\bf{v}\in\Sn$ we understand
by $D_{\bf{v}}$ the directional derivative in the direction
of $\bf{v}$.

Given Banach spaces $X$, $Y$, we write $\CalL(X,Y)$ for the
Banach space of bounded linear maps from $X$ into $Y$
endowed with the uniform operator topology (the norm
topology), and $\CalK(X,Y)$ for the set of compact
(completely continuous) operators in $\CalL(X,Y)$.  If
$X=Y$ we write simply $\CalL(X)$ and $\CalK(X)$.

For a metrizable topological space $S$ and a topological
vector space $Y$ the symbol $C(S,Y)$ denotes the vector
space of continuous functions from $S$ into $Y$.  If
$S\subset\Rn$, $C^{1}(S,Y)$ (resp.~$C^{2}(S,Y)$) stands for
the vector space of once (resp. twice) continuously
differentiable functions.  When $S$ is compact and $Y$ is a
Banach space, $C^{i}(S,Y)$, $C(S,Y)$, are understood to be
Banach spaces with the usual norms.  If $Y=\reals$ we write
$C^{i}(S)$, $C(S)$.

Throughout the paper $\CalA$ defines the differential operator
\begin{equation}
\label{operator}
{\CalA}u:=-\sum_{i,j=1}^{n}a_{ij}(x)D_{ij}u
-\sum_{i=1}^{n}a_{i}(x)D_{i}u,
\end{equation}
where $\Omega\subset\Rn$ is a bounded domain with boundary
$\domega$ of class $C^{2}$, $\clomega=\Omega\cup\domega$,
$a_{ij}=a_{ji}\in\Cetwo$,
$\sum_{i,j=1}^{n}a_{ij}(x)\xi_{i}\xi_{j}>0$ for each
$x\in\clomega$ and each nonzero
$(\xi_{1},\dots,\xi_{n})\in\Rn$, $a_{i}\in\Ceone$.
If the principal part of $\CalA$ is in the divergence form:
\begin{equation*}
{\CalA}u=-\sum_{i,j=1}^{n}D_{i}(a_{ij}(x)D_{j}u)
-\sum_{i=1}^{n}\tilde{a}_{i}(x)D_{i}u,
\end{equation*}
it suffices to assume $a_{ij}\in\Ceone$.

By $\CalB$ we denote the boundary operator
\begin{equation}
\label{boundary}
{\CalB}u:=u|_{\domega}.
\end{equation}
For $i=1$, $2$, define $C^{i}_{0}(\clomega):=\{u\in
C^{i}(\clomega):{\CalB}u=0\}$.  Similarly, $\Cezero:=\{u\in
Ce:{\CalB}u=0\}$. It is obvious that $C^{i}_{0}(\clomega)$
(\resp $\Cezero$) is a closed subspace of the Banach space
$C^{i}(\clomega)$ (\resp $\Ce$).

From now on we denote by $A$ the closure of $\CalA$ in
$\Ltwo$ and assume that $a_{0}\in
L^{\infty}(\reals\times\Omega)$ is fixed.  Put
\begin{equation*}
R:=\esssup_{t\in\reals,x\in\Omega}\abs{a_{0}(t,x)}.
\end{equation*}
It is well known that the closed ball in
$L^{\infty}(\reals\times\Omega)$ with center $0$ and radius
$R$ endowed with the weak$^{*}$ topology is a metrizable
compact space.

For $b\in L^{\infty}(\reals\times\Omega)$ and $t\in\reals$
by the {\em $t$\nobreakdash-\hspace{0pt}translate\/},
$b\cdot t$, of $b$ we understand the function $(b\cdot
t)(s,x):=b(t+s,x)$ for \pw $s\in\reals$ and \pw
$x\in\Omega$.  Define the {\em hull\/} $\Bb$ of $a_{0}$ as
the closure in the weak$^{*}$ topology of the set
$\{a_{0}\cdot t:t\in\reals\}$.

For $b\in\Bb$ and a measurable function $u:\Omega\to\reals$
set $\CalN(b)u:=b(0,\cdot)u$.  The multiplication operator
$\CalN(b)$ is easily seen to belong to $\CalL(\Ltwo)$.

By a {\em (mild) solution\/} to the parabolic partial
differential equation
\begin{equation}
\label{e0}
D_{t}u=\sum_{i,j=1}^{n}a_{ij}(x)D_{ij}u
+\sum_{i=1}^{n}a_{i}(x)D_{i}u+b(t,x)u,
\quad t>0, x\in\Omega,
\end{equation}
with the boundary condition
\begin{equation}
\label{e1}
u(t,x)=0, \quad t>0, x\in\domega,
\end{equation}
satisfying the initial condition
\begin{equation}
\label{e2}
u(0,x)=u_{0}(x),
\end{equation}
where $u_{0}\in\Ltwo$, and $b\in\Bb$ is considered a
parameter, we understand a function $u(\cdot;b,u_{0})\in
C([0,\infty),\Ltwo)$ satisfying the integral equation
\begin{equation}
\label{int-eq}
u(t)=e^{-At}u_{0}+
\int_{0}^{t}e^{-A(t-s)}{\Nbs}u(s)\,ds
\quad\text{for all $t\ge0$},
\end{equation}
where $\{e^{-At}\}_{t\ge0}$ is the holomorphic semigroup of
bounded linear operators on $\Ltwo$ generated by $(-A)$.

The proof of the next result may be patterned after the
proof of Thm.~3.3 in Chow, Lu and
Mallet-Paret~\cite{Ch-L-MP2} and Thm.~1.3 in
\JM~\cite{JM2}, and we do not present it here.
\begin{theorem}
\label{existence}
\begin{description}
\item{\em(i)} For each $b\in\Bb$ and $u_{0}\in\Ltwo$ there
exists a unique solution $u(\cdot;b,u_{0})$
to~(\ref{e0})+(\ref{e1})+(\ref{e2}).
\item{\em(ii)} For each $T>0$ the
mapping
\begin{equation*}
\Bb\times\Ltwo\ni(b,u_{0})\mapsto
u(\cdot;b,u_{0})\in C([0,T],\Ltwo)
\end{equation*}
is continuous.
\item{\em(iii)} For any $0<T_{1}\le T_{2}$ the mapping
\begin{equation*}
\Bb\ni b\mapsto u(\cdot;b,\bullet)\in
C([T_{1},T_{2}],\CalK(\Ltwo,\Cezeroone))
\end{equation*}
is continuous.
\end{description}
\end{theorem}

Recall that to (E)+(BC) we assign the {\em (formally)
adjoint equation\/}
\begin{equation*}
D_{t}v=-\sum_{i,j=1}^{n}D_{ij}(a_{ij}(x)v)
+\sum_{i=1}^{n}D_{i}(a_{i}(x)v)-a_{0}(t,x)v,
\quad t<0, x\in\Omega,
\tag{E$^{\sharp}$}
\end{equation*}
complemented with the homogeneous Dirichlet boundary
conditions
\begin{equation*}
\tag{BC$^{\sharp}$}
v(t,x)=0,\quad t<0, x\in\domega.
\end{equation*}

We define the {\em adjoint differential operator\/},
\begin{equation}
\label{adjoint}
{\CalA}^{\sharp}v:=-\sum_{i,j=1}^{n}D_{ij}(a_{ij}(x)v)
+\sum_{i=1}^{n}D_{i}(a_{i}(x)v),
\end{equation}
and the {\em adjoint boundary operator\/}
$\CalB^{\sharp}:=\CalB$.

Let $A'$ stand for the closure of $\CalA^{\sharp}$ in
$\Ltwo$.  For $b\in\Bb$ and $v_{0}\in\Ltwo$ we say that
$v(\cdot;b,u_{0})\in C((-\infty,0],\Ltwo)$ is a {\em (mild)
solution\/} to the equation
\begin{equation}
\label{esharp0}
D_{t}v=-\sum_{i,j=1}^{n}D_{ij}(a_{ij}(x)v)
+\sum_{i=1}^{n}D_{i}(a_{i}(x)v)-b(t,x)v,
\quad t<0, x\in\Omega,
\end{equation}
with the boundary condition
\begin{equation}
\label{esharp1}
v(t,x)=0, \quad t<0, x\in\domega,
\end{equation}
satisfying the initial condition
\begin{equation}
\label{esharp2}
v(0,x)=v_{0}(x),
\end{equation}
if it satisfies the integral equation
\begin{equation}
\label{adj-int-eq}
v(t)=e^{-A'(-t)}v_{0}+
\int_{t}^{0}e^{-A'(s-t)}{\Nbs}v(s)\,ds
\quad\text{for all $t\le0$},
\end{equation}
where $\{e^{-A't}\}_{t\ge0}$ is the holomorphic semigroup
of bounded linear operators on $\Ltwo$ generated by
$(-A')$.

Obviously we have the following counterpart to
Theorem~\ref{existence}:
\begin{theorem}
\label{adjoint-existence}
\begin{description}
\item{\em(i)} For each $b\in\Bb$ and $v_{0}\in\Ltwo$ there
exists a unique solution $v(\cdot;b,v_{0})$
to~(\ref{esharp0})+(\ref{esharp1})+(\ref{esharp2}).
\item{\em(ii)} For each $T<0$ the mapping
\begin{equation*}
\Bb\times\Ltwo\ni(b,v_{0})\mapsto
v(\cdot;b,v_{0})\in C([T,0],\Ltwo)
\end{equation*}
is continuous.
\item{\em(iii)} For any $T_{1}\le T_{2}<0$ the mapping
\begin{equation*}
\Bb\ni b\mapsto v(\cdot;b,\bullet)\in
C([T_{1},T_{2}],\CalK(\Ltwo,\Cezeroone))
\end{equation*}
is continuous.
\end{description}
\end{theorem}

We mention here an important property.
\begin{proposition}
\label{cocycle0}
\begin{description}
\item{\em(i)} For $t_{1}$, $t_{2}\ge0$, $b\in\Bb$ and
$u_{0}\in\Ltwo$ one has
\begin{equation*}
u(t_{1}+t_{2};b,u_{0})=
u(t_{1};b\cdot t_{2},u(t_{2};b,u_{0}))=
u(t_{2};b\cdot t_{1},u(t_{1};b,u_{0})).
\end{equation*}
\item{\em(i)} For $t_{1}$, $t_{2}\le0$, $b\in\Bb$ and
$v_{0}\in\Ltwo$ one has
\begin{equation*}
v(t_{1}+t_{2};b,v_{0})=
v(t_{1};b\cdot t_{2},v(t_{2};b,v_{0}))=
v(t_{2};b\cdot t_{1},v(t_{1};b,v_{0})).
\end{equation*}
\end{description}
\end{proposition}

For $b\in\Bb$ and $t>0$ define the linear
operator $\bar\psi(t,b)\in\CalL(\Ltwo,\Ceone)$ as
$$
\bar\psi(t,b)u_{0}:=u(t;b,u_{0}).
$$
and, for $b\in\Bb$ and $t<0$ define
$\bar\psi^{\sharp}(t,b)\in\CalL(\Ltwo,\Ceone)$ as
$$
\bar\psi^{\sharp}(t,b)v_{0}:=v(t;b,v_{0}).
$$

Occasionally we will consider solutions defined on open
intervals.  We say that $u\in C(J,\Ltwo)$, where
$J\subset\reals$ is an interval, is a {\em solution of
(\ref{e0})+(\ref{e1}) on $J$\/} if for any $T_{1}$,
$T_{2}\in J$, $T_{1}\le T_{2}$, one has
\begin{equation*}
u(T_{2})=\bar\psi(T_{2}-T_{1},b\cdot T_{1})u(T_{1}).
\end{equation*}
Solutions of (\ref{esharp0})+(\ref{esharp1}) on $J$ are
defined in an analogous way.

Put $\hat\psi(t,b):=i_{2}\circ\bar\psi(t,b)$,
$\hat\psi^{\sharp}(t,b):=
i_{2}\circ\bar\psi^{\sharp}(t,b)$, where $i_{2}$ denotes
the embedding $\Ceone\incmap\Ltwo$.

An important consequence of Proposition~\ref{cocycle0} is
the following {\em cocycle identity}.
\begin{equation}
\label{cocycle1}
\hat\psi(t_{1}+t_{2},b)=
\hat\psi(t_{1},b\cdot t_{2})\circ\hat\psi(t_{2},b),
\quad b\in\Bb, t_{1}\ge0, t_{2}\ge0,
\end{equation}
where we understand $\hat\psi(0,b)=\Id_{\Ltwo}$.

As the proof of the next proposition is based in the
standard way on Green's equality, we do not present it
here.
\begin{proposition}
\label{duality}
For $b\in\Bb$ and $t>0$ we have
$$
\hat\psi(t,b)^{*}=\hat\psi^{\sharp}(-t,\bt),
$$
where $^{*}$ represents the dual operator.
\end{proposition}

Henceforth, the symbol $[\cdot|\cdot]$ will stand for the
duality pairing between $\Cezeroone$ and $\Cezeroone^{*}$.
\begin{theorem}
\label{extension}
For $t>0$ and $b\in\Bb$ the operators $\bar\psi(t,b)$ and
$\bar\psi^{\sharp}(-t,\bt)$ extend respectively to the
operators $\psi(t,b)\in
{\CalK}(\Cezeroone^{*},\Cezeroone)$, $\psi(t,b)'\in
{\CalK}(\Cezeroone^{*},\Cezeroone)$.  These extensions are
continuous in $(t,b)\in(0,\infty)\times\Bb$.  Also,
\begin{equation*}
[\psi(t,b)u_{0}|v_{0}]=[\psi(t,b)'v_{0}|u_{0}]
\end{equation*}
for all $t>0$, $b\in\Bb$, $u_{0}$ and
$v_{0}\in\Cezeroone^{*}$.  (Consequently,
$\psi(t,b)^{*}=\psi(t,b)'$)
\end{theorem}
\begin{proof}
By the cocycle identity (\ref{cocycle1}) one has
\begin{equation}
\label{proof:ext}
\bar\psi(t,b)=
\bar\psi(t/2,b\cdot(t/2))\circ\hat\psi(t/2,b),
\end{equation}
with $\hat\psi(t/2,b)\in\CalK(\Ltwo)$ and
$\bar\psi(t/2,b\cdot(t/2))\in\CalK(\Ltwo,\Cezeroone)$.
Taking into account Proposition~\ref{duality} one obtains
\begin{align*}
&\hat\psi(t/2,b)=\hat\psi(t/2,b)^{**}\\
&=(i_{2}\circ\bar\psi^{\sharp}(-t/2,b\cdot(t/2)))^{*}
=\bar\psi^{\sharp}(-t/2,b\cdot(t/2))^{*}
\circ i_{2}^{*}.
\end{align*}
We define
\begin{equation*}
\psi(t,b):=
\bar\psi(t/2,b\cdot(t/2))\circ
\bar\psi^{\sharp}(-t/2,b\cdot(t/2))^{*}.
\end{equation*}
\eqref{proof:ext} yields
\begin{equation*}
\bar\psi(t,b)^{*}=
\hat\psi(t/2,b)^{*}\circ\bar\psi(t/2,b\cdot(t/2))^{*}.
\end{equation*}
Again by Proposition~\ref{duality} one obtains
\begin{equation*}
\hat\psi(t/2,b)^{*}=
\hat\psi^{\sharp}(-t/2,b\cdot(t/2)))=
i_{2}\circ\bar\psi^{\sharp}(-t/2,b\cdot(t/2))).
\end{equation*}
We set
\begin{equation*}
\psi(t,b)':=
\bar\psi^{\sharp}(-t/2,b\cdot(t/2)))\circ
\bar\psi(t/2,b\cdot(t/2))^{*}.
\end{equation*}
The assertion on duality follows by the definition.  The
continuous dependence is a consequence of
Theorems~\ref{existence} and~\ref{adjoint-existence} and
the continuity of the composition of linear operators in
the uniform operator topology.
\end{proof}

As we shall see now, an easy consequence of the above
theorem is that $\bar\psi(t,b)$ is an integral operator
with a fairly regular kernel $G(t,b)(\cdot,\bullet)$ (the
{\em Green's function\/}).

We write $\delta_{x}$ for the Dirac delta at
$x\in\clomega$, and $\delta'_{(x,\boldv)}$ for the
directional derivative operator at $x\in\clomega$ in the
direction of $\boldv\in\Sn$.  Both $\delta_{x}$ and
$\delta'_{(x,\boldv)}$ are regarded as elements of
$\Cezeroone^{*}$.

For $t>0$, $b\in\Bb$, $x\in\clomega$ set
\begin{equation}
\label{def-of-Green-f1}
G(t,b)(x,\cdot):=\psi(t,b)^{*}\delta_{x},
\end{equation}
and
\begin{equation}
\label{def-of-Green-f2}
G^{*}(t,b)(x,\cdot):=\psi(t,b)\delta_{x}.
\end{equation}
For $u_{0}$ and $v_{0}$ in $\Ltwo$ one has
\begin{equation}
\label{Green-f}
\begin{split}
(\psi(t,b)u_{0})(x)&=[\psi(t,b)u_{0}|\delta_{x}]\\
&=[\psi(t,b)^{*}\delta_{x}|u_{0}]
=\int_{\Omega}G(t,b)(x,\xi)\,u_{0}(\xi)\,d\xi
\end{split}
\end{equation}
and
\begin{equation}
\label{Green-f-adj}
\begin{split}
(\psi(t,b)^{*}v_{0})(x)
&=[\psi(t,b)^{*}v_{0}|\delta_{x}]\\
&=[\psi(t,b)\delta_{x}|v_{0}]
=\int_{\Omega}G^{*}(t,b)(x,\xi)\,v_{0}(\xi)\,d\xi.
\end{split}
\end{equation}
\begin{lemma}
\label{lemma:Ce-one-Green}
The mappings
\begin{equation*}
(0,\infty)\times\Bb\times
\clomega\ni(t,b,x)\mapsto G(t,b)(x,\cdot)
\in\Cezeroone
\end{equation*}
and
\begin{equation*}
(0,\infty)\times\Bb\times
\clomega\ni(t,b,x)\mapsto G^{*}(t,b)(x,\cdot)
\in\Cezeroone
\end{equation*}
are continuous.
\end{lemma}
\begin{proof}
We prove the result only for $G$, the proof for $G^{*}$
being similar.  Let $t_{k}\to t>0$, $b_{k}\to b$ and
$x_{(k)}\to x$.  In~particular, $\delta_{x_{(k)}}$
converges weakly-* to $\delta_{x}$.  One has
\begin{equation*}
\begin{split}
&\Ceonenorm{G(t_{k},b_{k})(x_{(k)},\cdot)-
G(t,b)(x,\cdot)}
=\Ceonenorm{\psi(t_{k},b_{k})^{*}\delta_{x_{(k)}}-
\psi(t,b)^{*}\delta_{x}}\\
\le{}&\Ceonenorm{\psi(t_{k},b_{k})^{*}
\delta_{x_{(k)}}-\psi(t,b)^{*}\delta_{x_{(k)}}}
+\Ceonenorm{\psi(t,b)^{*}\delta_{x_{(k)}}-
\psi(t,b)^{*}\delta_{x}}.
\end{split}
\end{equation*}
The first term is estimated by
\begin{equation*}
\Ceonenorm{(\psi(t_{k},b_{k})^{*}-
\psi(t,b)^{*})\delta_{x_{(k)}}}
\le\lVert\psi(t_{k},b_{k})^{*}-
\psi(t,b)^{*}
\rVert_{\CalL(\Cezeroone^{*},\Cezeroone)},
\end{equation*}
which tends, by Theorem~\ref{extension}, to $0$ as
$k\to\infty$.  One proves that the second term approaches
$0$ by noting that the dual $\psi(t,b)^{*}$ of the compact
operator $\psi(t,b)$ takes weakly-* convergent sequences
into (norm) convergent ones.
\end{proof}

\begin{lemma}
\label{Green-f-duality}
For each $t>0$, $b\in\Bb$, $x$, $\xi\in\clomega$ we have
$G^{*}(t,b)(x,\xi)=G(t,b)(\xi,x)$.
\end{lemma}
\begin{proof}
Let $u_{0}$, $v_{0}\in\Ce$.  One has
\begin{align*}
[\psi(t,b)u_{0}|v_{0}]&=
\int_{\clomega}(\psi(t,b)u_{0})(x)\,v_{0}(x)\,dx\\
&=\int_{\clomega}
\biggl(\int_{\clomega}G(t,b)(x,\xi)\,u_{0}(\xi)\,d\xi\biggr)
v_{0}(x)\,dx\\
\intertext{and}
[\psi(t,b)^{*}v_{0}|u_{0}]&=
\int_{\clomega}(\psi(t,b)^{*}v_{0})(x)\,u_{0}(x)\,dx\\
&=\int_{\clomega}u_{0}(x)
\biggl(\int_{\clomega}G^{*}(t,b)(x,\xi)\,
v_{0}(\xi)\,d\xi\biggr)\,dx.
\end{align*}
As $[\psi(t,b)u_{0}|v_{0}]=[\psi(t,b)^{*}v_{0}|u_{0}]$, we
get
\begin{equation*}
\iint_{\clomega\times\clomega}
G^{*}(t,b)(x,\xi)\,u_{0}(x)\,v_{0}(\xi)\,dx\,d\xi=
\iint_{\clomega\times\clomega}
G(t,b)(\xi,x)\,u_{0}(x)\,v_{0}(\xi)\,dx\,d\xi,
\end{equation*}
whence we derive our assertion in the standard way.
\end{proof}
From the definition and Lemma~\ref{Green-f-duality} it
follows directly that $G$ is a solution of
(\ref{e0})+(\ref{e1}) in $x$ on $(0,\infty)$ and a solution
of the adjoint equation (\ref{esharp0})+(\ref{esharp1}) in
$\xi$ on $(0,\infty)$.

By
\begin{equation*}
\frac{\partial}{\partial\boldv_{x}}G(t,b)(x,\xi)\quad
\text{(\resp}
\frac{\partial}{\partial\boldv_{\xi}}G(t,b)(x,\xi)
\text{)}
\end{equation*}
we denote the directional derivative of $G$ in $x$ (\resp
in $\xi$) in the direction of $\boldv\in\Sn$.  The next
result shows that the above derivatives are solutions of
the corresponding equations on $(0,\infty)$.
\begin{proposition}
\label{deriv-of-Green-f}
\begin{description}
\item{\em(i)} For $t>0$, $b\in\Bb$, $x\in\clomega$
and $\boldv\in\Sn$ we have
\begin{equation*}
\frac{\partial}{\partial\boldv_{x}}G(t,b)(x,\cdot)
=\psi(t,b)^{*}\delta'_{(x,\boldv)}.
\end{equation*}
\item{\em(ii)} For $t>0$, $b\in\Bb$, $\xi\in\clomega$
and $\boldv\in\Sn$ we have
\begin{equation*}
\frac{\partial}{\partial\boldv_{\xi}}G(t,b)(\cdot,\xi)
=\psi(t,b)\delta'_{(\xi,\boldv)}.
\end{equation*}
\end{description}
\end{proposition}
\begin{proof}
We consider only (i) as the case (ii) is treated similarly.
It suffices to prove that
\begin{equation}
\label{proof:deriv-of-Green-f}
[\frac{\partial}{\partial\boldv_{x}}G(t,b)(x,\cdot)|u_{0}]
=[\psi(t,b)^{*}\delta'_{(x,\boldv)}|u_{0}]
\end{equation}
for each $u_{0}\in\Ce$.  The right-hand side of
\eqref{proof:deriv-of-Green-f} can be written as
\begin{equation*}
[\psi(t,b)^{*}\delta'_{(x,\boldv)}|u_{0}]
=[\psi(t,b)u_{0}|\delta'_{(x,\boldv)}]
=\frac{\partial}{\partial\boldv_{x}}
\left(\int_{\clomega}G(t,b)(x,\xi)\,u_{0}(\xi)\,d\xi\right).
\end{equation*}
Because by Lemma~\ref{lemma:Ce-one-Green}
$(\partial/{\partial}\boldv_{x})G(t,b)$ is a continuous
function of $(x,\xi)\in\clomega\times\clomega$, we can
interchange differentiation and integration in the
right-hand expression.  But obviously
\begin{equation*}
\int_{\clomega}\frac{\partial}{\partial\boldv_{x}}
G(t,b)(x,\xi)\,u_{0}(\xi)\,d\xi
=[u_{0}|\frac{\partial}{\partial\boldv_{x}}G(t,b)(x,\xi)].
\end{equation*}
\end{proof}
The arguments of the proof of Lemma~\ref{lemma:Ce-one-Green}
easily extend to show the following.
\begin{lemma}
\label{lemma:Green-mixed-deriv}
The mappings
\begin{equation*}
(0,\infty)\times\Bb\times
\clomega\times\Sn\ni(t,b,x,\boldv)\mapsto
(\partial/{\partial}\boldv_{x})G(t,b)(x,\cdot)
\in\Cezeroone
\end{equation*}
and
\begin{equation*}
(0,\infty)\times\Bb\times\clomega\times\Sn
\ni(t,b,\xi,\boldv)\mapsto
(\partial/{\partial}\boldv_{\xi})G(t,b)(\cdot,\xi)
\in\Cezeroone
\end{equation*}
are continuous.
\end{lemma}

\section{Positive solutions}
\label{Positive}
For a Banach space $X$ consisting of real functions (\resp
of equivalence classes, modulo measure zero sets, of real
functions) defined on $\clomega$, by $X_{+}$ we denote the
{\it (standard) cone\/} of nonnegative functions (\resp of
equivalence classes of functions nonnegative a.e.).  It is
straightforward that $X_{+}$ is a closed convex set such
that

a) For each $u\in X_{+}$ and $\alpha\ge0$ one has
${\alpha}u\in X_{+}$, and

b) $X_{+}$ does not contain any one-dimensional subspace.

If $X=\Ltwo$, or $X=\Ce$, or $X=\Ceone$, the cone $X_{+}$
is {\it reproducing\/}, that is, $X_{+}+X_{+}=X$.

For $X=\Ltwo$ or $X=\Ce$, let $\le$ denote the partial ordering
induced on $X$ by the cone $X_{+}$:
\begin{equation*}
v\le u\quad\text{if }u-v\in X_{+}.
\end{equation*}
We write $v<u$ if $v\le u$ and $v\ne u$.  The reverse
symbols are used in the standard way.

The cone $\Cezero_{+}$ (\resp $\Cezeroone_{+}$) consists of
those $u\in\Cezero$ (\resp $u\in\Cezeroone$) for which
$u(x)\ge0$ at each $x\in\clomega$.  The cone $\Cezero_{+}$
has empty interior, whereas it is easy to verify that the
interior of $\Cezeroone_{+}$ equals
\begin{equation*}
\Cezeroone_{++}:=\{u\in\Cezeroone:u(x)>0\text{ for
$x\in\Omega$ and $D_{\nu}u(x)<0$ for $x\in\domega$}\},
\end{equation*}
where $D_{\nu}$ is the derivative in the direction of
the unit normal vector field $\nu:\domega\to\Rn$ pointing
out of $\Omega$.  The partial ordering induced on
$\Cezeroone$ by $\Cezeroone_{+}$ is denoted by $\leone$.
For $u$, $v\in\Cezeroone$ we write $u\lessone v$ if
$u\leone v$ and $u\ne v$, and $u\llone v$ if
$v-u\in\Cezeroone_{++}$.

Let $\e$ be the (nonnegative) principal eigenfunction of
the Laplacian on $\Omega$ with Dirichlet boundary
conditions normalized so that $\Cenorm{\e}=1$.  The
standard regularity theory and the maximum principle for
elliptic PDE's yield the following.
\begin{lemma}
\label{l:prop-of-e}
{\em(i)} $\e\in\Ceone$.
\par{\em(ii)} $\e(x)>0$ for $x\in\Omega$.
\par{\em(iii)} $\e(x)=0$ and $D_{\nu}\e(x)<0$
for $x\in\domega$.
\par
(In other words, $\e\in\Cezeroone_{++}$.)
\end{lemma}

For $u\in\Ce$ define the {\em order-unit norm\/} of $u$ as
$$
\Ceenorm{u}:=\inf\{\alpha>0:-{\alpha}\e(x)\le
u(x)\le{\alpha}\e(x)\text{ for each }x\in\clomega\}.
$$
By Amann~\cite{Am1}, the set
$$
\Cee:=\{u\in\Ce:\Ceenorm{u}<\infty\}
$$
is a Banach space with the norm $\Ceenorm{\cdot}$,
partially ordered by the relation $\lee$ induced by the
cone $\Cee_{+}:=\Cee\cap\Ce_{+}$ with nonempty interior
$\Cee_{++}$.  Evidently $\Cee$ embeds (set-theoretically)
in $\Cezero$.  For $u$, $v\in\Cee$ we write $u\Le v$ if
$u\lee v$ and $u\ne v$, and $u\LL v$ if $v-u\in\Cee_{++}$.

The order-unit norm $\Ceenorm{\cdot}$ is {\it monotonic\/},
that is, $0\lee u\lee v$ implies
$\Ceenorm{u}\le\Ceenorm{v}$.

Since, as can be easily verified from the definition,
$\Cenorm{u}\le\Ceenorm{u}$ for each $u\in\Cee$, the
embedding of $\Cee$ into $\Cezero$ is continuous. Moreover,
it embeds the cone $\Cee_{+}$ into the cone $\Cezero_{+}$.

Before proceeding further, we introduce a concept that will
be useful in the sequel.  Define the mapping
$g:\domega\times[0,1)\to\Rn$ by the formula
$g(x,\tau):=x-\tau\nu(x)$.  This mapping is easily seen to
be of class $C^{1}$.  Further, for each $x\in\domega$ its
derivative at $(x,0)$ is an isomorphism.  By the Inverse
Function Theorem there is a neighborhood $U$ of
$\domega\times\0$ in $\domega\times[0,1)$ such that
$g|_{U}$ is a $C^{1}$ diffeomorphism onto its image $V$.
Taking $V$ smaller if necessary we can assume
$V\subset\clomega$.  Now we define
\begin{equation*}
r:=\pi_{1}\circ g^{-1},
\end{equation*}
where $\pi_{1}$ is the projection onto the first
coordinate.  It is straightforward that $r$ is a $C^{1}$
retraction of $V$ onto $\domega$.

\begin{proposition}
\label{propo:embed}
\begin{description}
\item{\em(i)} The natural embedding $i:\Cezeroone\incmap\Cee$ is
continuous.
\item{\em(ii)} $i(\Cezeroone_{++})\subset\Cee_{++}$.
\end{description}
\end{proposition}
\begin{proof}
To prove (i), suppose by way of contradiction that there is
a sequence $(u_{k})_{k=1}^{\infty}\subset\Cezeroone$ such
that $\Ceonenorm{u_{k}}=1$ and $\Cenorm{u_{k}}>k$ for each
$k\in\naturals$.  Passing to a subsequence if necessary we
can assume that $u_{k}$ converge in the
$\Ce$\nobreakdash-norm to some
$\tilde{u}\in\Cezero$.  By the definition of $\Cee$, for
each $k\in\naturals$ there is $x_{(k)}\in\clomega$ such
that $\abs{u_{k}(x_{(k)})}>k\e(x_{(k)})$.  Obviously, no
such $x_{(k)}$ can belong to $\domega$.  Further, there is
no subsequence of $(x_{(k)})$ converging to
$\tilde{x}\in\Omega$, since otherwise
$\abs{\tilde{u}(\tilde{x})}\ge k\e(\tilde{x})$ for each
$k\in\naturals$, which contradicts $\e(\tilde{x})>0$.  Now
we take a subsequence (denoted again by
$(x_{(k)})$) converging to
$\bar{x}\in\domega$.  The continuity of the retraction $r$
implies $\lim_{k\to\infty}r(x_{(k)})=r(\bar{x})=\bar{x}$.
For each sufficiently large $k$ we have
\begin{equation*}
\frac{\e(x_{(k)})-\e(r(x_{(k)}))}{\norm{x_{(k)}-r(x_{(k)})}}
<\frac{1}{k}\frac{\abs{u_{k}(x_{(k)})-u_{k}(r(x_{(k)}))}}
{\norm{x_{(k)}-r(x_{(k)})}}.
\end{equation*}
The expression on the left-hand side tends, as
$k\to\infty$, to $-D_{\nu}\e(\bar{x})$, which is positive
by Lemma~\ref{l:prop-of-e}.  The Mean Value Theorem tells
us that for each $k\in\naturals$ there is
$\bar{x}_{(k)}\in\Omega$ belonging to the line segment with
endpoints $x_{(k)}$ and $r(x_{(k)})$, such that
\begin{equation*}
u_{k}(x_{(k)})-u_{k}(r(x_{(k)}))=
-\frac{{\partial}u_{k}}{\partial\nu_{r(x_{(k)})}}
(\bar{x}_{(k)})
\cdot\norm{x_{(k)}-r(x_{(k)})}.
\end{equation*}
As the derivatives of all $u_{k}$'s are bounded, this
yields $D_{\nu}\e(\bar{x})=0$, a contradiction.

Assume now that $u\in\Cezeroone_{++}$.  It is
straightforward that $u\in\Cee_{+}$.  In order to prove
that $u\in\Cee_{++}$ it is enough to show that there is
$K>0$ such that $u(x)\ge K\e(x)$ for each $x\in\clomega$.
Suppose not.  Then there is a sequence
$(x_{(k)})_{k=1}^{\infty}$ such that
$u(x_{k})<k\e(x_{(k)})$ for each $k\in\naturals$.
Evidently each $x_{(k)}$ belongs to $\Omega$.  As in the
proof of (i) we prove that each convergent subsequence of
$(x_{(k)})_{k=1}^{\infty}$ tends to some
$\bar{x}\in\domega$, from which it follows that
$D_{\nu}\e(\bar{x})=0$.
\end{proof}

We proceed now to investigating order-preserving properties
of $\psi(t,b)$.  The following theorem is a consequence of
the parabolic strong maximum principle (compare problems
3.3.6--3.3.8 in Henry's book~\cite{Henry}).
\begin{theorem}
\label{p:mono}
\begin{description}
\item{\em(i)} For $t>0$, $b\in\Bb$ and a nonzero
$u_{0}\in\Ltwo_{+}$ we have
$\psi(t,b)u_{0}\in\Cezeroone_{++}$.
\item{\em(i)} For $t>0$, $b\in\Bb$ and a nonzero
$v_{0}\in\Ltwo_{+}$ we have
$\psi(t,b)^{*}v_{0}\in\Cezeroone_{++}$.
\end{description}
\end{theorem}

For $t>0$, $b\in\Bb$ put
\begin{equation*}
\psi_{\e}(t,b):=i\circ\psi(t,b)\circ I,
\end{equation*}
where $i:\Cezeroone\incmap\Cee$,
$I:\Cee\incmap\Cezeroone^{*}$ are the natural embeddings.
The {\em cocycle identity\/} (see (\ref{cocycle1})) takes
the form
\begin{equation}
\label{cocycle2}
\psi_{\e}(t_{1}+t_{2},b)=
\psi_{\e}(t_{1},b\cdot t_{2})\circ\psi_{\e}(t_{2},b),
\quad b\in\Bb, t_{1}>0, t_{2}>0.
\end{equation}
Set $\CalW$ to be the product Banach bundle $\Bb\times\Cee$
with base space $\Bb$ and model fiber $\Cee$.  We define an
endomorphism $\Psi_{\e}$ of the bundle $\CalW$ by the
formula
\begin{equation*}
\Psi_{\e}(b,u)=(b\cdot1,\psi_{\e}(b)u), \quad b\in\Bb,
u\in\Ce,
\end{equation*}
where $\psi_{\e}(b):=\psi_{\e}(1,b)$.

The mapping $\Psi_{\e}$ is an endomorphism of the Banach
bundle $\CalW$, covering the homeomorphism $\phi$,
$\phi(b)=b\cdot1$, of the base space $\Bb$.  Its iterates
$\Psi_{\e}^{k}$, $k\in\naturals$, form the {\em linear
skew-product semidynamical system\/}
$\{\Psi_{\e}^{k}\}_{k=1}^{\infty}$ on $\CalW$.  For
$k\in\naturals$ we write
$$
\Psi_{\e}^{k}(b,u)=(b\cdot k,\psi_{\e}^{(k)}(b)u),
$$
where we denote
$$
\psi_{\e}^{(k)}(b):=\psi_{\e}(b\cdot(k-1))
\circ\psi_{\e}(b\cdot(k-2))\circ\dots
\circ\psi_{\e}(b\cdot1)\circ\psi_{\e}(b).
$$

The bundle endomorphism $\Psi_{\e}^{*}$ dual to $\Psi_{\e}$,
$$
\Psi_{\e}^{*}(b,v):=(b\cdot(-1),\psi_{\e}(b)^{*}v),
\quad b\in\Bb,\,v\in\Cee^{*},
$$
acts on the dual bundle $\CalW^{*}=\Bb\times\Cee^{*}$ and
covers the inverse homeomorphism $\phi^{-1}$,
$\phi^{-1}(b)=b\cdot(-1)$.  Its iterates form the linear
skew-product dynamical system
$\{(\Psi_{\e}^{*})^{k}\}_{k=1}^{\infty}$ {\em dual\/} to
$\{\Psi_{\e}^{k}\}_{k=1}^{\infty}$.

As a consequence of Theorem~\ref{p:mono} we have that
$\psi_{\e}(b)u\GG0$ for each $b\in\Bb$ and each $u\in\Cee$,
$u\Ge0$.  We refer to this property as the {\em strong
monotonicity\/} of the linear skew-product semidynamical
system $\{\Psi_{\e}^{k}\}$.

The next theorem is based on results of P. \Polacik\ and I.
\Terescak~\cite{Po-Te} (for an earlier result compare the
present author's unpublished manuscript~\cite{JM1}).  We
begin with introducing some notation: for a subbundle
$\CalW_{1}$ of $\CalW$, we denote
$\CalW_{1}(b)=\{u\in\Ce:(b,u)\in\CalW_{1}\}$ (in other
words, $\{b\}\times\CalW_{1}(b)$ is the fiber of
$\CalW_{1}$ over $b\in\Bb$).  A subbundle $\CalW_{1}$ is
called {\em invariant\/} if for each $b\in\Bb$ from
$u\in\CalW_{1}(b)$ it follows that
$\psi_{\e}(b)u\in\CalW_{1}(b\cdot1)$.

\begin{theorem}
\label{th:exponential-separation}
There exists a decomposition of $\CalW$ into a direct sum
$\CalW=\CalS\oplus\CalT$ of invariant subbundles having the
following properties:
\begin{description}
\item{\em(i)} There is a continuous mapping $\Bb\ni
b\mapsto w_{b}\in\Cee_{++}$ such that $\Ceenorm{w_{b}}=1$
and $\CalS(b)=\{{\alpha}w_{b}:\alpha\in\reals\}$.
In~particular, $\CalS$ has dimension one and
$\CalS(b)\setminus\0\subset\Cee_{++}\cup-\Cee_{++}$ for any
$b\in\Bb$.
\item{\em(ii)} There is a continuous mapping $\Bb\ni
b\mapsto w^{*}_{b}\in\Cee_{++}$ such that
$\Ceenorm{w^{*}_{b}}=1$ and $\CalT(b)=
\{u\in\Cee:\int_{\clomega}u(x)\,w^{*}_{b}(x)\,dx=0\}$.
In~particular, $\CalT$ has codimension one and
$\CalT(b)\cap\Cee_{+})=\0$ for any $b\in\Bb$.
\item{\em(iii)} The mapping $\Psi_{\e}|_{\CalS}$ is a
bundle automorphism.
\item{\em(iv)} There are constants $D\ge1$ and $0<\lambda<1$
such that
$$
\frac{\Ceenorm{\psi_{\e}^{(k)}(b)u_{1}}}
{\Ceenorm{\psi_{\e}^{(k)}(b)u_{2}}}\le
D{\lambda}^{k}
\frac{\Ceenorm{u_{1}}}{\Ceenorm{u_{2}}}
$$
for each $b\in\Bb$, $u_{1}\in\CalT(b)$,
$u_{2}\in\CalS(b)\setminus\0$,
$k\in\naturals$.
\end{description}
\end{theorem}
\begin{proof}
According to~\cite{Po-Te}, $\CalW$ decomposes into a direct
sum of invariant subbundles $\CalS$ and $\CalT$ satisfying
(i), (iii) and (iv).  Moreover, there is a continuous
mapping $\Bb\ni b\mapsto\tilde{w}^{*}_{b}\in\Cee^{*}$ such
that
\begin{description}
\item{(a)} $\Ceestnorm{\tilde{w}^{*}_{b}}=1$,
\item{(b)} $[u|\tilde{w}^{*}_{b}]_{\e}>0$ for each
$b\in\Bb$ and each nonzero $u\in\Cee_{+}$,
\item{(c)} $\CalT(b)=
\{u\in\Cee:[u|\tilde{w}^{*}_{b}]_{\e}=0\}$,
\item{(d)} for each $b\in\Bb$ there is $\kappa(b)>0$ such
that $\psi_{\e}^{*}(b)\tilde{w}^{*}_{b}=\kappa(b)
\tilde{w}^{*}_{b\cdot(-1)}$,
\end{description}
where $[\cdot|\cdot]_{\e}$ denotes the duality pairing
between $\Cee$ and $\Cee^{*}$.

From (d) we derive by Theorem~\ref{extension} that
$\tilde{w}^{*}_{b}\in\Cee$ for each $b\in\Bb$. Property~(b)
implies $\tilde{w}^{*}_{b}\in\Cee_{+}\setminus\0$.  By
Theorem~\ref{p:mono} and (d),
$\tilde{w}^{*}_{b}\in\Cee_{++}$ for each $b\in\Bb$.
Putting $w^{*}_{b}:=
\tilde{w}^{*}_{b}/\Ceenorm{\tilde{w}^{*}_{b}}$ completes
the proof.
\end{proof}

Define the linear operator $P(b)\in{\CalL}(\Cee)$ as
$$
P(b)u:=[u|w^{*}_{b}]_{\e}w_{b}=
\biggl(\int_{\clomega}u(x)\,w^{*}_{b}(x)\,dx\biggr)w_{b}.
$$
It is easy to verify that $P(b)$ has kernel $\CalT(b)$,
range $\CalS(b)$, and $P(b)\circ P(b)=P(b)$.  Hence $P(b)$
is a projector of $\Cee$ onto $\CalS(b)$.  By the
continuity of $w_{b}$ and $w^{*}_{b}$ in $b$ it follows
that the assignment $\Bb\ni b\mapsto P(b)\in{\CalL}(\Cee)$
is continuous.  The family $P:=\{(b,P(b)):b\in\Bb\}$ is a
bundle projection with kernel $\CalT$ and range $\CalS$.

The property described in
Theorem~\ref{th:exponential-separation}(iv) is called {\em
exponential separation\/}.  With the help of the projection
$P$ we can write it as
\begin{equation}
\label{formula:exp-sep}
\frac{\Ceenorm{(\Id-P(b\cdot k))\psi_{\e}^{(k)}(b)u}}
{\Ceenorm{P(b\cdot k)\psi_{\e}^{(k)}(b)u}}\le
D{\lambda}^{k}
\frac{\Ceenorm{(\Id-P(b))u}}{\Ceenorm{P(b)u}}
\end{equation}
for $b\in\Bb$, $u\in\Cee\setminus\CalT(b)$,
$k\in\naturals$.

The following theorem is crucial in proving our main
results in the next subsection.
\begin{theorem}
\label{p:projection}
There is a constant $L>0$ such that for each $b\in\Bb$ and
each $u\in\Cee_{+}$ we have
$$
\frac{\Ceenorm{(\Id-P(b\cdot1))\psi_{\e}(b)u}}
{\Ceenorm{P(b\cdot1)\psi_{\e}(b)u}}\le L.
$$
\end{theorem}
Before proving the above theorem we need a series of
auxiliary results.
\begin{proposition}
\label{p:green-mono}
\begin{description}
\item{\em(i)} $G(t,b)(x,\xi)>0$ for each
$(x,\xi)\in\Omega\times\Omega$ and each
$t>0$, $b\in\Bb$.
\item{\em(ii)} $G(t,b)(x,\xi)=0$ for each
$(x,\xi)\in(\Omega\times\domega)\cup
(\domega\times\Omega)$ and each $t>0$, $b\in\Bb$.
\item{\em(iii)} $(\partial/\partial\nu_{x})G(t,b)(x,\xi)<0$
for each $(x,\xi)\in\domega\times\Omega$, $t>0$ and
$b\in\Bb$.
\item{\em(iv)} $(\partial/\partial\nu_{\xi})G(t,b)(x,\xi)<0$
for each $(x,\xi)\in\Omega\times\domega$, $t>0$ and
$b\in\Bb$.
\item{\em(v)} $(\partial/\partial\nu_{x})G(t,b)(x,\xi)=
(\partial/\partial\nu_{\xi})G(t,b)(x,\xi)=0$ for each
$(x,\xi)\in\domega\times\domega$, $t>0$ and $b\in\Bb$.
\item{\em(vi)} $(\partial^{2}/
\partial\nu_{x}\partial\nu_{\xi})G(t,b)(x,\xi)<0$
for each $(x,\xi)\in\domega\times\domega$, $t>0$
and $b\in\Bb$.
\end{description}
\end{proposition}
\begin{proof}
This is a consequence of (\ref{def-of-Green-f1}),
(\ref{def-of-Green-f2}), Lemma~\ref{Green-f-duality},
Proposition~\ref{deriv-of-Green-f},
Lemma~\ref{lemma:Green-mixed-deriv} and
Theorem~\ref{p:mono}.
\end{proof}

For $b\in\Bb$ define the function
$\barGb:\clomega\to\Cezeroone$ as
$$
\barGb(\xi):=G(1,b)(\cdot,\xi),
$$
and the function $\Gb:\clomega\to\Cee$ as $\Gb:=i\circ\barGb$
(recall that $i$ denotes the embedding
$\Cezeroone\incmap\Cee$).
\begin{lemma}
\label{l:incl-mono}
\begin{description}
\item{\em(i)} The assignment
\begin{equation*}
\Bb\ni b\mapsto\Gb\in C^{1}(\clomega,\Cee)
\end{equation*}
is continuous.
\item{\em(ii)} $\Gb(\xi)\GG 0$ for each $b\in\Bb$ and
$\xi\in\Omega$.
\item{\em(iii)} $(\partial/\partial\nu)\Gb(\xi)\LL 0$
for each $b\in\Bb$ and $\xi\in\domega$.
\end{description}
\end{lemma}
\begin{proof}
The corresponding properties hold for the function $\barGb$
by Lemmas~\ref{lemma:Ce-one-Green}, \ref{Green-f-duality}
and~\ref{lemma:Green-mixed-deriv} (part~(i)), and
Proposition~\ref{p:green-mono} (parts (ii) and (iii)).  Now
it suffices to make use of Proposition~\ref{propo:embed}.
\end{proof}

For $b\in\Bb$ and $\xi\in\Omega$ we define
$$
m(b,\xi):=\sup\{\alpha\ge0:\Gb(\xi)\gee\alpha\e\}.
$$
It is easily checked that $m$ is a lower semicontinuous
function of $(b,\xi)\in\Bb\times\Omega$.
\begin{lemma}
\label{l:focusing}
The number
$$
\gamma:=\inf\biggl\{\frac{m(b,\xi)}{\Ceenorm{\Gb(\xi)}}:
b\in\Bb,\xi\in\Omega\biggl\}
$$
is positive.
\end{lemma}
\begin{proof}
Recall that $r$ is a retraction of a neighborhood $V$ of
$\domega$ in $\clomega$ onto $\domega$.  By taking the
neighborhood $V$ smaller (if necessary) we can assume that
$$
\frac{\partial}{\partial\nu_{r(\xi)}}\Gb(\xi)\LL0\quad
\text{for all $\xi\in V$}.
$$
As $m$ is lower semicontinuous, there is $\epsilon_{1}>0$
such that
$$
\frac{\partial}{\partial\nu_{r(\xi)}}\Gb(\xi)\LL
-\epsilon_{1}\e\quad\text{for all $b\in\Bb$ and all $\xi\in
V$}.
$$
By integration we get
$\Gb(\xi)\GG\epsilon_{1}\norm{\xi-r(\xi)}$, therefore
$m(b,\xi)\ge\epsilon_{1}\norm{\xi-r(\xi)}$.

On the other hand, by continuity there is $\epsilon_{2}>0$
such that
$$
\frac{\partial}{\partial\nu_{r(\xi)}}\Gb(\xi)\GG
-\epsilon_{2}\e\quad\text{for all $b\in\Bb$ and all $\xi\in
V$},
$$
from which it follows that
$\Ceenorm{\Gb(\xi)}\le\epsilon_{2}\norm{\xi-r(\xi)}$.
Consequently,
$$
\frac{m(b,\xi)}{\Cenorm{\Gb(\xi)}}\ge
\frac{\epsilon_{1}}{\epsilon_{2}}\quad\text{for all
$b\in\Bb$ and $\xi\in V\setminus\domega$}.
$$
Now, as $\Gb(\xi)\GG0$ for $\xi$ in the compact set
$\Omega\setminus\interior_{\clomega}V$ and the
function $m$ is lower semicontinuous and positive on
$\Omega\setminus\interior_{\clomega}V$, there is
$\gamma_{1}>0$ such that
$$
\frac{m(b,\xi)}{\Cenorm{\Gb(\xi)}}\ge
\gamma_{1}\quad\text{for all $b\in\Bb$ and
$\textstyle\xi\in\Omega\setminus\interior_{\clomega}V$}.
$$
We have thus proved that
$\gamma\ge\min\{\epsilon_{1}/\epsilon_{2},\gamma_{1}\}>0$.
\end{proof}

\begin{proposition}
\label{p:focusing}
For each $b\in\Bb$ and $u\in\Cee_{+}$ we have
$$
\gamma\Ceenorm{\psi_{\e}(b)u}\e\lee\psi_{\e}(b)u
\lee\Ceenorm{\psi_{\e}(b)u}\e.
$$
\end{proposition}
\begin{proof}
The inequality $\psi_{\e}(b)u\lee\Ceenorm{\psi(b)u}\e$ follows by
the definition of the $\Cee$\nobreakdash-\hspace{0pt}norm.
Further, one has
$$
\psi_{\e}(b)u
=\int_{\clomega}u(\xi)\,\Gb(\xi)\,d\xi
\gee\biggl(\int_{\clomega}m(b,\xi)\,u(\xi)
\,d\xi\biggr)\e,
$$
and, by Lemma~\ref{l:focusing}
$$
\biggl(\int_{\clomega}m(b,\xi)\,u(\xi)\,d\xi\biggr)\e
\gee\gamma
\biggl(\int_{\clomega}
\Ceenorm{\Gb(\xi)}u(\xi)\,d\xi\biggr)\e,
$$
whereas
$$
\int_{\clomega}\Ceenorm{\Gb(\xi)}u(\xi)\,d\xi
=\int_{\clomega}\Ceenorm{u(\xi)\,\Gb(\xi)}\,d\xi
\ge\Ceenorm{\psi_{\e}(b)u},
$$
from which our assertion follows.
\end{proof}
The reader familiar with the theory of order-preserving
operators in ordered Banach spaces will notice that
Proposition~\ref{p:focusing} implies that (Hilbert's)
projective distance between $\e$ and $\psi_{\e}(b)u$ is not
larger than $-\log\lambda$, uniformly in $b\in\Bb$ and
$u\in\Cee_{+}\setminus\0$.  We could therefore say that the
family of linear operators $\{\psi_{\e}(b):b\in\Bb\}$ is
{\em uniformly focusing\/} (for more on Hilbert's
projective metric see \eg Nussbaum's monographs~\cite{N1},
\cite{N2}).  However, I choose to prove the results
directly rather than use those concepts.

\begin{proof}[Proof of Theorem~\ref{p:projection}]
By Proposition~\ref{p:focusing} and the definition of $P$ we
have
\begin{equation}
\label{p:projection:proof}
\begin{split}
\Ceenorm{P(b\cdot1)\psi_{\e}(b)u}&=
\int_{\clomega}
(\psi_{\e}(b)u)(x)\,w^{*}_{b\cdot1}(x)\,dx\\
&\ge\gamma\biggl(
\int_{\clomega}\e(x)\,w^{*}_{b\cdot1}(x)\,dx\biggl)
\Ceenorm{\psi_{\e}(b)u}.
\end{split}
\end{equation}
For $b\in\Bb$ set $\mu(b):=\sup\{\alpha\ge0:
w^{*}_{b\cdot1}\gee\alpha\e\}$.  As $w^{*}_{b\cdot1}\GG0$,
we have $\mu(b)>0$ for all $b\in\Bb$.  The function
$b\mapsto\mu(b)$ is easily seen to be lower semicontinuous,
hence $\mu:=\inf\{\mu(b),b\in\Bb\}$ is positive.  By
\eqref{p:projection:proof} we obtain
\begin{equation*}
\Ceenorm{P(b\cdot1)\psi_{\e}(b)u}\ge
\gamma\mu\Ceenorm{\psi_{\e}(b)u}.
\end{equation*}
The assertion of the theorem now follows by noting that the
operator norms of $\Id-P(b)$ are bounded above.
\end{proof}

A function $u:\reals\times\clomega\to\reals$ is called a
{\em global solution\/} to (E)+(BC) if $t\mapsto
u(t,\cdot)$ is a solution of (E)+(BC) on
$(-\infty,\infty)$. As a consequence of
Theorem~\ref{existence}, for a global solution $u$ the
function $\reals\ni t\mapsto u(t,\cdot)\in\Cezeroone$ is
continuous.  We will refer to a global solution $u$ as {\em
nontrivial\/} if $u(t,\cdot)\ne0$ for any $t\in\reals$.
Obviously the set of global solutions forms a vector
subspace of $C(\reals,\Cezeroone)$.

A nontrivial global solution $u$ is said to be {\em
globally positive\/} if $u(t,\cdot)\in\Cee_{+}$ for all
$t\in\reals$.  {\em Globally negative\/} solutions are
defined in an analogous way.  Denote by $\CalP$ the vector
space consisting of all globally positive and all globally
negative solutions of (E)+(BC).  In~particular, the null
solution belongs to $\CalP$.
\begin{proposition}
\label{globally-positive}
If $u\in\CalP$ then $u(t,\cdot)\in\CalS(a_{0}\cdot t)$ for
all $t\in\reals$.
\end{proposition}
\begin{proof}
Suppose by way of contradiction that there is
$\tau\in\reals$ such that
$u(\tau,\cdot)\notin\CalS(a_{0}\cdot\tau)$.  By
invariance of $\CalS$ one has
$u(\tau-k,\cdot)\notin\CalS(a_{0}\cdot(\tau-k))$ for
any $k\in\naturals$.  Theorem~\ref{p:projection} gives
\begin{equation*}
\frac
{\Ceenorm{(\Id-P(a_{0}\cdot(\tau-k)))u(\tau-k,\cdot)}}
{\Ceenorm{P(a_{0}\cdot(\tau-k))u(\tau-k,\cdot)}}\le L
\quad\text{ for all $k\in\naturals\cup\0$}.
\end{equation*}
For the remainder of the proof we drop the symbol of the
base point in the projection $P$.  Put
$M:=\Ceenorm{(\Id-P)u(\tau,\cdot)}
/\Ceenorm{Pu(\tau,\cdot)}$.  Take a positive integer
$k_{0}$ not less than $\log(M/2DL)/\log\lambda$.
Exponential separation yields
\begin{equation*}
\begin{split}
\frac{\Ceenorm{(\Id-P)u(\tau,\cdot)}}
{\Ceenorm{Pu(\tau,\cdot)}}&=
\frac
{\Ceenorm{\psi_{\e}(k_{0},a_{0}\cdot(\tau-k_{0}))
(\Id-P)u(\tau-k_{0},\cdot)}}
{\Ceenorm{\psi_{\e}(k_{0},a_{0}\cdot(\tau-k_{0}))
Pu(\tau-k_{0},\cdot)}}\\
&\le
D{\lambda}^{k_{0}}
\frac{\Ceenorm{(\Id-P)u(\tau-k_{0},\cdot)}}
{\Ceenorm{Pu(\tau-k_{0},\cdot)}}\\
&\le D{\lambda}^{k_{0}}L\le\frac{M}{2},
\end{split}
\end{equation*}
a contradiction.
\end{proof}
Let $\gothP:\CalP\to\CalS(a_{0})$ be the linear operator
\begin{equation*}
{\gothP}u:=u(0,\cdot).
\end{equation*}

\begin{theorem}
\label{main-theorem}
$\gothP$ is a linear isomorphism.  Therefore, $\dim\CalP=1$.
\end{theorem}
\begin{proof}
Let $u\in\CalP$ be such that $u(0,\cdot)=0$.  Obviously
$u(t,\cdot)=0$ for all $t\ge0$.  Suppose to the contrary
that $u(\tau,\cdot)\ne0$, say
$u(\tau,\cdot)\in\Cee_{+}\setminus\0$, for some $\tau<0$.
By Theorem~\ref{p:mono} $u(0,\cdot)\GG0$, a contradiction.
This proves that $\gothP$ is a monomorphism.  To prove that
$\gothP$ is an epimorphism it suffices to find a globally
positive solution $\bar{u}$ such that
$\bar{u}(0,\cdot)=w_{a_{0}}$. For $k\in\naturals$ we define
\begin{align*}
\bar{u}(k,\cdot)&:=\psi_{\e}^{(k)}(a_{0})w_{a_{0}},\\
\bar{u}(-k,\cdot)&:=
(\psi_{\e}^{(k)}(a_{0}\cdot(-k))|_{\CalS(a_{0}\cdot(-k))})^{-1}
w_{a_{0}\cdot(-k)}.
\end{align*}
As a result of the cocycle identity (\ref{cocycle2})
$\bar{u}(k_{2},\cdot)=
\psi_{\e}(k_{2}-k_{1},a_{0}\cdot k_{1})\bar{u}(k_{1},\cdot)$
for any two integers $k_{1}<k_{2}$.  Now, set
\begin{equation*}
\bar{u}(t,\cdot):=
\psi_{\e}(t-[t],a_{0}\cdot[t])\bar{u}([t],\cdot)
\end{equation*}
for each $t\in\reals\setminus\integers$. Observe that for
such a $t$ we have (again by the cocycle identity)
\begin{equation*}
\begin{split}
\bar{u}([t]+1,\cdot)&=
\psi_{\e}([t]-t+1,a_{0}\cdot t)\bar{u}(t,\cdot)\\
&=\psi_{\e}([t]-t+1,a_{0}\cdot t)
\psi_{\e}(t-[t],a_{0}\cdot[t])\bar{u}([t],\cdot)\\
&=\psi_{\e}(1,a_{0}\cdot[t])\bar{u}([t],\cdot),
\end{split}
\end{equation*}
hence
\begin{equation*}
\bar{u}(t_{2},\cdot)=\psi_{\e}(t_{2}-t_{1},a_{0}\cdot
t_{1})\bar{u}(t_{1},\cdot)
\end{equation*}
for any $t_{1}<t_{2}$.
\end{proof}

\section*{Appendix: Other boundary conditions}
Here we briefly outline how the results presented in the
main body of the paper carry over to some other boundary
conditions (at~least for the case that the principal part
of $\CalA$ is in the divergence form).

Assume that the boundary $\domega$ of $\Omega$ is a
disjoint union of two closed sets, $\domega=\dD\cup\dR$,
where $\dR$ is of class $C^{3}$.  The boundary conditions
will be the following
\begin{equation}
\tag{BC'}
\begin{aligned}
u(t,x)&=0,\quad t\in\reals,\, x\in\dD,\\
\frac{{\partial}u}{\partial\beta}(t,x)+c(x)u(t,x)&=0,
\quad t\in\reals,\, x\in\dR,
\end{aligned}
\end{equation}
where $\beta\in C^{1}(\dR,\Rn)$ is a nontangential vector
field pointing out of $\Omega$, and $c\in C^{1}(\dR)$ is a
nonnegative function.  We assume for the sake of
definiteness that $\dD\ne\emptyset$, since otherwise the
results can be obtained in a much simpler way by working in
the Banach space $L^{1}(\Omega)$, as shown by the author
in~\cite{JM2}.

The boundary operator $\CalB$ is defined by
\begin{equation*}
{\CalB}u(x):=
\begin{cases}
u(x)&\text{for $x\in\dD$},\\
\frac{{\partial}u}{\partial\beta}(x)+c(x)u(x)&
\text{for $x\in\dR$}.
\end{cases}
\end{equation*}
In place of $\Cezero$ we take $C_{\CalB}(\clomega)$ defined
as $C_{\CalB}(\clomega):=\{u\in\Ce:u(x)=0$ for $x\in\dD\}$.
The symbols $C^{i}_{\CalB}(\clomega)$, $i=1$, $2$, are
defined in an analogous way.

The first of the major modifications concerns the definition
of the adjoint boundary operator $\CalB^{\sharp}$
\begin{equation*}
{\CalB^{\sharp}}v(x):=
\begin{cases}
v(x)&\text{for $x\in\dD$},\\
\frac{{\partial}v}{\partial\beta^{\sharp}}(x)+
c^{\sharp}(x)v(x)&\text{for $x\in\dR$},
\end{cases}
\end{equation*}
where the nontangential vector field $\beta^{\sharp}$ and
the nonnegative function $c^{\sharp}$ are chosen so as for
the Green's identity to be fulfilled.  For explicit
formulas the reader is referred to Amann~\cite{Am2}.

In Section~\ref{Positive} the function $\e$ must be defined
as the (normalized) nonnegative principal eigenvalue of the
Laplacian on $\Omega$ with the boundary condition
${\CalB}u=0$.  Lemma~\ref{l:prop-of-e} takes the form
\begin{LA}
\label{appendix:l:prop-of-e}
{\em(i)} $\e\in\Ceone$.
\par{\em(ii)} $\e(x)>0$ for $x\in\Omega\cup\dR$.
\par{\em(iii)} $\e(x)=0$ and $D_{\nu}\e(x)<0$
for $x\in\dD$.
\end{LA}
Accordingly, in further results $\Omega$ (\resp $\domega$)
should be replaced by $\Omega\cup\dR$ (\resp $\dD$).


\begin{thebibliography}{99}

\bibitem{Am1} H. Amann, Fixed point equations and nonlinear
eigenvalue problems in ordered Banach spaces, {\em SIAM
Rev.\/} \textbf{18} (1976), 620--709.

\bibitem{Am2} H. Amann, Dual semigroups and second order
linear elliptic boundary value problems, {\em Israel J.
Math.\/} \textbf{45} (1983), 225--254.

\bibitem{Ch-L-MP1} S.-N. Chow, K. Lu and J. Mallet-Paret,
Floquet bundles for parabolic differential equations, {\em
\JDE\/} \textbf{109} (1994), 147--200.

\bibitem{Ch-L-MP2} S.-N. Chow, K. Lu and J. Mallet-Paret,
Floquet bundles for scalar parabolic equations, {\em Arch.
Rational Mech. Anal.\/} \textbf{29} (1995), 245--304.

\bibitem{Henry} D. Henry, ``Geometric Theory of Semilinear
Parabolic Equations,'' {\em in:\/} Lecture Notes in Math.,
\textbf{840}, Springer, Berlin--New York, 1981.

\bibitem{JM1} J. \JM, Flows on ordered bundles, preprint.

\bibitem{JM2} J. \JM, Globally positive solutions of linear
parabolic PDEs of second order with Robin boundary
conditions, {\em \JMAA\/} \textbf{209} (1995), 47--59.

\bibitem{N1} R. D. Nussbaum, ``Hilbert's Projective Metric
and Iterated Nonlinear Maps,'' {\em in:\/} Mem. Amer. Math.
Soc. \textbf{75}(391), 1988.

\bibitem{N2} R. D. Nussbaum, ``Hilbert's Projective Metric
and Iterated Nonlinear Maps, II,'' {\em in:\/} Mem. Amer.
Math. Soc. \textbf{79}(401), 1989.

\bibitem{Po-Te} P. \Polacik\ and I. \Terescak, Exponential
separation and invariant bundles for maps in ordered Banach
spaces with applications to parabolic equations, {\em
\JDDE\/} \textbf{5} (1993), 279--303.
\end{thebibliography}
\end{document}